\documentclass[12pt]{amsart}
\usepackage{amscd,amssymb,amsmath,verbatim,amsthm, graphicx}

\usepackage{hyperref}

\usepackage{color}

\newcommand{\bS}{\mathbb S}

\newcommand{\cW}{\mathcal{W}}
\newcommand{\tcW}{\widetilde{\cW}}
\newcommand{\wAH}{\widehat{AH}}
\newcommand{\bom}{\mathbf{\omega}}

\newcommand{\CC}{\mathbb C}

\newcommand{\RR}{\mathbb R}

\newcommand{\PP}{\mathbb P}

\newcommand{\calM}{{\mathcal M}}

\newcommand{\calO}{{\mathcal O}}

\newcommand{\calU}{{\mathcal U}}

\newtheorem{theorem}{Theorem}

\newtheorem{corollary}{Corollary}

\theoremstyle{definition}
\newtheorem{definition}{Definition}
\theoremstyle{remark}
\newtheorem*{remark}{Remark}

\title[]{A gluing construction of $D_{k}$ ALF gravitational instantons and existence of non-holomorphic minimal spheres} 

\author{Xuwen Zhu}
\address{Northeastern University, Boston, MA 02115}
\email{x.zhu@northeastern.edu}
\date{}

\begin{document}

\maketitle

\begin{abstract}
This note extends the construction of $D_{k}$ ALF gravitational instantons in Schroers--Singer~\cite{SS} to a new case where the nonlinear superposition  is given by the $D_{1}$ Atiyah--Hitchin metric and $k-1$ copies of $A_{0}$ Taub-NUT metrics.  We then give a general class of ALF spaces such that each of them contains a non-holomorphic minimal sphere. Together with Foscolo's construction~\cite{Foscolo} this gives a large class of $K3$ surfaces containing non-holomorphic minimal spheres. 
\end{abstract}

\section{Introduction}	
A gravitational instanton \footnote{The term "gravitational instanton" is debated and might mean a different object for other authors. In this paper we only discuss the hyperK\"ahler case.} is a complete noncompact real 4-dimension hyperK\"ahler manifold with $L^{2}$ curvature decay at infinity. There has been a long history in the  study of gravitational instantons and a lot of recent developments have been made. Gravitational instantons arise as building blocks for degeneration of K3 surfaces~\cite{GW, Foscolo, HSVZ,CVZ}, which makes them powerful tools in probing $K3$ surfaces near degeneration limits.

The gravitational instantons have been classified into six types, ALE/
ALF/ALG/ALH/ALG*/ALH*, based on the end structure at infinity~\cite{Minerbe, CC, SunZhang}. In this paper we will discuss the ALF gravitational instantons, which is modeled  at infinity by quotients of circle bundles over $\RR^{3}$  and has cubic volume growth. There are two families of ALF spaces depending on the monodromy at infinity being cyclic or dihedral, $A_{k}$ and $D_{k}$, labelled by non-negative integers $k$. In this paper we will discuss the  construction of $D_{k}$ type ALF spaces via a nonlinear superposition of $D_{1}$ and $k-1$ copies of $A_{0}$ ALF spaces, and using this we will show the existence of non-holomorphic minimal spheres in these spaces as well as in $K3$ surfaces. 

This gluing construction extends the work by Schroers and Singer \cite{SS}, where the authors constructed $D_{k}$ ALF gravitational instantons by gluing the moduli space $\calM_{2}^{0}$ of centered SU(2) monopoles (which is the only $D_{0}$ ALF space) with $k$ copies of Taub-NUT space (each of which is an $A_{0}$ ALF space). As commented in their paper, this construction can be extended to other cases. In this paper we work on the case where $\calM_{2}^{0}$  is replaced by the Atiyah--Hitchin metric AH~\cite{AH}, which is defined to be the double cover of $\calM_{2}^{0}$ and is a rotationally invariant $D_{1}$ space. Gluing the AH metric with $k-1$ copies of Taub-NUT space (one copy fewer than~\cite{SS}), we get a $D_{k}$ space.  A rough statement of the theorem is given below and a more explicit description will be given at Theorem~\ref{thm:main}.
\begin{theorem}
Given any $k\geq 2$ and any set of $k$ distinct points in $\RR^{3}$ denoted as $\{0, p_{1}, \dots, p_{k-1}\}$ , there exists $\epsilon_{0}>0$, such that for any $\epsilon\in (0,\epsilon_{0})$ there exists a $D_{k}$ ALF metric $g_{\epsilon}$ which is a nonlinear superposition of the Atiyah--Hitchin (AH) metric centered at 0 and $k-1$ copies of Taub-NUT metrics (TN) centered at $p_{i}/\epsilon, 1\leq i\leq k-1$. 
\end{theorem}

There has been a lot of progress in the analysis and construction of ALF spaces~\cite{MR2511656, Minerbe, CC2, BiMi}. The $A_{k}$ type ALF has been well understood as a multi-Taub-NUT space via gluing construction~\cite{Minerbe2, GH}. And $D_{k}$ spaces have been constructed via twistor methods~\cite{MR0757213} and desingularization~\cite{BiMi, Auvray, MR1231230}.
The construction in~\cite{SS} gives an explicit gluing constructions of $D_{k}$ type ALF based on Sen's proposal~\cite{Sen}, and the construction follows the physical interpretation of modeling electrically charged particles like protons and electrons~\cite{AMS}. There are also works by Cherkis, Kapustin, Hitchin, Roček et.al~\cite{Cherkis1, Cherkis2, Cherkis3, LR, IR} describing ALF spaces as moduli space of monopoles with prescribed singularities.  The construction in~\cite{SS} gave $D_{k}$ ALF when the $k$ Taub-NUT  singularities are well separated from the origin, while the ones constructed in this paper can be viewed as the limit of these $D_{k}$ spaces when one of these Taub-NUT singularities approaches zero.

It can be potentially further generalized, as suggested in~\cite{SS}, to cases of composing a $D_{j}$ ALF space with various $A_{\nu_{i}-1}$ ALF to generate other  $D_{k'}$ spaces, where $k'=j+\sum \nu_{i}$, which would give other limiting cases when more TN singularities travel to zero or to each other.  Since in this paper we are mainly interested in using the property of Atiyah--Hitchin space, we will only work on the case where the center piece is exactly given by AH.

Our construction follows the same procedure and the analysis is done following~\cite{SS}, as the change of the center piece from $\calM_{0}^{2}$ to $AH$ does not change the asymptotic behavior. The mapping properties of such metrics can be analyzed using the fiber boundary calculus, or $\phi-$calculus~\cite{MMphi, Vail, HHM}, which captures the asymptotic metric structure of ALF spaces.  

As a consequence of the construction, we obtain a family of $D_{k}$ spaces that, with properly chosen marked points, degenerate to the Atiyah--Hitchin space as the Gromov Hausdorff limit. One special feature of the AH space is that it contains a minimal sphere that is not holomorphic with any of the complex structures in the hyperK\"ahler rotation~\cite{MW2}. In the same paper it was shown that this minimal sphere is stable, i.e. the second variation of the area functional is nonnegative. In fact this sphere is shown to be "strongly stable" and area minimizing by Tsai and Wang~\cite{TsaiWang}. This special sphere was also recently discussed in the compactness of Fueter sections~\cite{EL}. 
 With our gluing construction we can obtain such special sphere in all nearby $D_{k}$ space.

\begin{theorem}[same as Theorem~\ref{thm2}]
For any $k\geq 2$, there is a family of $D_{k}$ ALF metrics $\{g_{\epsilon}\}_{\epsilon\in (0,\epsilon)}$, such that each of them contains a stable minimal sphere that is not holomorphic with respect to any of the compatible complex structures in the hyperK\"ahler rotation.  
\end{theorem}

There has been a lot of recent developments relating to minimal surfaces in the context of gravitational instantons. Lotay and Oliveira~\cite{LotayOliveira} classified all $\bS^{1}$-invariant minimal surfaces in Gibbons--Hawking ansatz.
Foscolo and Trinca~\cite{FoscoloTrinca} recently constructed new non-holomorphic minimal 2-spheres in multi-Taub--NUT spaces and K3 surfaces using Scherk surface and holomorphic cigars, yet those 2-spheres are unstable. For more general discussion regarding (non-)holomorphic minimal surfaces, we also refer to~\cite{Arezzo, Donaldson, MW1}. 

In~\cite{Foscolo} Foscolo constructed a collapsing sequence of K3 surfaces using ALF spaces as part of the building blocks.  As a corollary he used Atiyah--Hitchin space as the building block to show the existence of minimal non-holomorphic sphere in some K3 surface near the degeneration limit. Now with more ALF spaces constructed above with non-holomorphic spheres, we extend this result of Foscolo and show that there is a neighborhood in the moduli space of K3 surfaces such that each of these has a special sphere. 

\begin{corollary}[same as Corollary~\ref{cor1}]
There is an open neighborhood in the moduli space of hyperK\"ahler metrics on $K3$, such that each metric in this neighborhood contains a stable minimal sphere that is not holomorphic with respect to any compatible complex structures. 
\end{corollary}

This corollary is connected to one interpretation of the Weak Gravity Conjecture~\cite{Vafa}, which predicts the existence of ``non-BPS'' particles. In mathematics  such particles can be interpreted as non-holomorphic submanifolds in Calabi--Yau manifolds in the large complex structure limit~\cite{MR4204112, MR4090116, MR4580537}. The corollary above can be seen as one of the examples.

The paper is organized in the following order: Section 2 describes the two building blocks, the multi-Taub-NUT space and the Atiyah--Hitchin metric, then gives the gluing construction and perturbative argument to find the hyperK\"ahler metric. Section 3 proves the existence of non-holomorphic spheres in some $D_{k}$ ALF spaces and $K3$ surfaces.

\vskip 0.1in

\noindent\textbf{Acknowledgement: }The author would like to thank Lorenzo Foscolo, Rafe Mazzeo, Michael Singer and Saman Habibi Esfahani for helpful discussions, and the anonymous referee for the careful reading and valuable feedback. The author is partially supported by NSF DMS-2305363.

\section{The gluing construction of $D_{k}$ ALF spaces}
The gluing construction is a modification from~\cite{SS}, and we describe the construction below. The basic idea is to first construct an adiabatic Gibbons--Hawking Ansatz which is a singular circle fibration over $\RR^3$ outside a compact set and has $2(k-1)$ $\mathbb{Z}_2$-symmetric Taub--NUT singularities $\{\pm p_i/\epsilon\}_{i=1}^{k-1}\subset \RR^3$. Then  we glue in a double branched cover of  Atiyah--Hitchin metric, which is also a circle fibration, into the compact part of $\RR^3$.
After the quotient of the involution action $x\rightarrow -x$ in $\RR^3$, this approximate metric gives 
 the correct $D_{k}$ asymptotic topology at infinity (given by the multi-Taub--NUT metrics) and is the smooth Atiyah--Hitchin metric near 0. However the metric is only approximate hyperK\"ahler. Finally we solve for the actual hyperK\"ahler metric by perturbation of the hyperK\"ahler triple. 
 
 The set up is similar to~\cite{SS} where the center piece is a double branched cover of $\calM_{2}^{0}$. We also modify the coefficients in the adiabatic Gibbons--Hawking ansatz so that it is compatible with the $AH$ metric. Except those changes, the asymptotic behavior of $\calM_{2}^{0}$ and Atiyah--Hitchin metric are similar since  $AH$ is the double cover of $\calM_{2}^{0}$, therefore the analysis used for perturbing the approximate triple to get the exact solution can be applied directly.

Now we describe the two pieces needed for gluing.

\subsection{The multi-Taub-NUT space $M_{\epsilon}$}
Fix $\delta \in (0,1/2)$, and let $\epsilon \in (0,\delta^{2})$ be a small parameter which will go to 0 in the limit later. Fix a set of distinct points $p_{1}, \dots, p_{k-1} \in \{x\in \RR^{3}: |x|>2\epsilon\}$, we consider a family of 4-manifolds $M_{\epsilon}$, each of which is a singular circle fibration over $\{x\in \RR^{3}: |x|>2\epsilon\}$. We consider
the following adiabatic family of Gibbon--Hawking ansatz on $M_{\epsilon}$
\begin{equation}\label{eq:ge}
g_{\epsilon}=h_{\epsilon} \frac{|dx|^{2}}{\epsilon^{2}}+h_{\epsilon}^{-1}\alpha^{2},  \ \alpha= *_{\epsilon} dh_{\epsilon}
\end{equation}
where $\alpha$ is a connection 1-form of the circle bundle, and $*_{\epsilon}$ is the Hodge star with respect to the flat metric $\frac{|dx|^{2}}{\epsilon^{2}}$ on the base $\RR^{3}_{|x|>2\epsilon}$; the singular harmonic function $h_{\epsilon}$ on $\RR^{3}_{|x|>2\epsilon}$ is given by
\begin{equation}\label{eq:he}
h_{\epsilon}=1-\frac{\epsilon}{|x|} +\sum_{v=1}^{k-1} \left(\frac{\epsilon}{2|x+p_{v}|}+\frac{\epsilon}{2|x-p_{v}|}\right).
\end{equation}

We denote $\phi$ to be the singular fibration associated with the Gibbons--Hawking ansatz~\eqref{eq:ge} by
\begin{equation}
\phi: M_{\epsilon} \rightarrow \RR^{3}_{|x|> 2\epsilon}
\end{equation}
and note the fibration is only singular at $P:=\{0, \pm p_{1}, \dots, \pm p_{k-1}\}$. 
 This ansatz is modified from~\cite[equation (1.4)]{SS} by replacing the coefficients of $\frac{1}{|x|}$ from $2\epsilon$ to $\epsilon$ which will match the  Atiyah--Hitchin metric in the compact part later.  Note that, by a change of variables $\tilde x=x/\epsilon, q_v=p_v/\epsilon$, the harmonic function $h_\epsilon$ can be written as 
 $$
 h_{\epsilon}=1-\frac{1}{|x|} +\sum_{v=1}^{k-1} \left(\frac{1}{2|x+q_{v}|}+\frac{1}{2|x-q_{v}|}\right)
 $$
 which gives a Gibbon--Hawking ansatz with Taub--NUT singularities at $q_v$, therefore this matches the name ``multi-Taub--NUT space'' for $A_{k-1}$ ALF spaces.
 
The involution action of $\iota: x\mapsto -x$ on $\RR^{3}$ applies to the fibration. Note that the asymptotic of $h_{\epsilon}$ is given by
\begin{equation}
h_{\epsilon}=1+\frac{\epsilon(2k-4)}{2|x|} + O(\epsilon|x|^{-3}), \ |x|\gg 1
\end{equation}
which, after quotient by the involution $\iota$, gives the asymptotic topology of a dihedral $D_{k}$ ALF.

Such behavior is characterized by ``fiber boundary structures''~\cite{MR1734130}, and we recall the definition below. We will be working with a manifold $X$ with boundary $\partial X$, and a boundary defining function $\rho$ for $\partial X$ is a smooth function on $X$ such that $\rho|_{\partial X}=0$ and $d\rho$ is nowhere vanishing on $\partial X$. 
\begin{definition}
A compact manifold with boundary $(X,g)$ has a fiber boundary structure, or $\phi$-structure, if there is a fibration $\phi: \partial X \rightarrow Y$ where $Y$ is a compact manifold without boundary and $F$ is the fiber. The space of $\phi$-vector fields $\mathcal{V}_{\phi}$ is given by b-vector fields (i.e. they are tangent to $\partial X$) which also satisfy
$$
v|_{\partial X}\in C^{\infty}(\partial X, T(\partial X/Y)),
$$
and
$$v(\rho) \in \rho^{2}C^{\infty}(X)$$
where $\rho$ is a boundary defining function for $\partial X$.
\end{definition}
Locally using coordinates $\phi(y_{1},\dots, y_{b}, z_{1}, \dots, z_{f})=(y_{1}, \dots, y_{b})$, the $\phi$-vector fields are spanned by
$$
\rho^{2}\frac{\partial}{\partial \rho}, \rho\frac{\partial}{\partial y_{1}}, \dots, \rho\frac{\partial}{\partial y_{b}}, \frac{\partial}{\partial z_{1}}, \dots, \frac{\partial}{\partial z_{f}}.
$$
And a $\phi$-metric is a smooth metric on the $\phi$-tangent bundle $T_{\phi}X$,  locally given by
$$
g_{\phi}=\frac{d\rho^{2}}{\rho^{4}} + \frac{\phi^{*} g_{Y}}{\rho^{2}} + g_{F}.
$$

The fibration $\phi: M_{\epsilon}\rightarrow \RR^{3}_{|x|>2\epsilon}$ can be lifted to the radially-compactified base $\overline{\RR^{3}}\setminus\{|x|\leq 2\epsilon\}$. Here the radial compactification of $\RR^{3}$ is obtained by adding a sphere $\bS^{2}=\partial \overline{\RR^{3}}$ at infinity. Using $x\in \RR^{3}$ as the usual coordinates, the compactification introduces the inverted polar coordinates $\rho=1/|x|, y=\frac{x}{|x|} \in \bS^{2}$ near the boundary, and the metric $g_{\epsilon}$  in~\eqref{eq:ge} lifts to a fiber boundary metric of the form
$$
g_{\epsilon}=h_{\epsilon}(\frac{d\rho^{2}}{\rho^{4}}+\frac{|dy|^{2}}{\rho^{2}}) + h_{\epsilon}^{-1}\bar \alpha^{2}, 
$$
where $h_{\epsilon}$ and $\bar\alpha=*_{\epsilon}d h_{\epsilon}$ lift to be smooth on $\overline{\RR^{3}}\setminus\{|x|\leq 2\epsilon\}$ (one can check this by changing variables in~\eqref{eq:he}). A detailed description was given in~\cite[Proposition 3.4]{SS}.

\subsection{The branched cover $\wAH_{\epsilon}$}
We will glue the (double branched cover) of Atiyah--Hitchin metric into the center. Following the notation in~\cite{SS}, we denote by $\widehat{AH}$ the space $\CC\PP^{1}\times \CC\PP^{1} \setminus \text{anti-diagonal}$, or in coordinates: 
\begin{equation}
\widehat{AH}=\{(z,w) \in \CC^{2} \times \CC^{2}: \ z_{1}\bar w_{1} + z_{2} \bar w_{2}\neq 0\}/\CC^{*}\times \CC^{*}
\end{equation}
This space can be identified via the space of X- and Y-ellipsis (see~\cite[Appendix A]{SS} for details). In the region near infinity, one can use coordinates $\{X\in \CC^{3}|\sum_{i=1}^{3} X_{i}^{2}=1\}$. If we write $X=\tilde x + i\xi, \ \tilde x, \xi \in \RR^{3}$, then they satisfy
$$|\tilde x|^{2} =1+|\xi|^{2},  \ \tilde x \cdot \xi=0.$$
There are two involution maps on $\widehat{AH}$:
$$
\begin{aligned}
s: (\tilde x, \xi) \mapsto (-\tilde x, -\xi)\\
r: (\tilde x, \xi) \mapsto (\tilde x, -\xi).
\end{aligned}
$$
The relation between $\widehat{AH}$, the Atiyah--Hitchin space AH, and the $D_{0}$ monopole moduli space $\mathcal{M}_{2}^{0}$ is given by 
\begin{equation}
AH=\widehat{AH}/s, \ \mathcal{M}_{2}^{0} = AH/r. 
\end{equation}
The relation implies that near infinity $\widehat{AH}$ is a circle bundle $\psi: \widehat{AH} \rightarrow \RR^{3}\setminus\{|x'|>1\}$
 with the following asymptotic
\begin{equation}\label{eq:AH}
g_{\widehat{AH}} = (1-\frac{1}{|x'|})|dx'|^{2} + (1-\frac{1}{|x'|})^{-1}(\alpha')^{2} + O(e^{-|x'|})
\end{equation}
where $x'$ is bijectively related to $\tilde x$ in an implicit way. In particular,  $AH$ and its branched cover $\widehat{AH}$ are both strongly ALF spaces in the sense of~\cite[Definition 3.5]{SS}. Recall a  $\phi$-metric $g$ is called ``strongly ALF'' if it is close to an exact Gibbons-Hawking ansatz of $A_{k}$ or $D_{k}$ ALF space $g_{ALF}$ modulo an error vanishing to infinite order, i.e. outside a compact set $|g-g_{ALF}|=\calO(\rho^{\infty})$ where $\rho$ is a boundary defining function for $g$ with respect to the $\phi$-structure. In particular, both $AH$ and $\widehat{AH}$ satisfy this condition, with an exponentially decaying error.
Having this condition ensures the asymptotic analysis at infinity can be applied the same way. As pointed out in~\cite[Theorem C.4 and Remark C.5]{SS}, when $g$ is strongly ALF, one can construct a ``nicer'' parametrix to solve the Poisson equation $\Delta_{g}u=f$. When the right hand side $f$ is smooth and decaying, one can get a smooth decaying solution $u$. On the other hand for a generic ALF metric one only expects polyhomogeneity rather than smoothness for $u$ in this case. This is used in the second step of the proof of Theorem~\ref{thm:main}, where the linearized operator is given by 12 copies of scalar Laplace operators, and one needs to iteratively solve away the error and make sure the solution is still smooth. 

To connect the asymptotics of $g_{\widehat{AH}}$ with the adiabatic Gibbons--Hawking ansatz constructed in~\eqref{eq:ge},  we modify~\eqref{eq:AH} to add the Taub-NUT singularities and get
\begin{equation}
g_{\widehat{AH},\epsilon}=h_{\epsilon}' |dx'|^{2} + (h_{\epsilon}')^{-1} (\alpha')^{2}
\end{equation}
where 
$$
h_{\epsilon}'=1-\frac{1}{|x'|} +\sum_{v=1}^{k-1}\frac{\epsilon}{|p_{v}|}, \ \alpha'=*dh_{\epsilon}'.$$ 
Here $|x'|>\frac{1}{\delta}$ for $\delta \ll 1$.  
We call the associated circle fibration total space $\wAH_{\epsilon}$ and the fibration
\begin{equation}
\psi: \wAH_{\epsilon} \rightarrow \RR^{3}\setminus\{|x'|>1/\delta\}.
\end{equation}
Later we will connect the large neck region $\{\frac{1}{\delta} < |x'|< \frac{\delta}{\epsilon}\}$ of $\wAH_{\epsilon}$ to the region of $\{\frac{\epsilon}{\delta} <|x|<\delta\}$ in the multi-Taub-NUT $M_{\epsilon}$.

 With the same compactification at infinity of $\RR^{3}$, we can see that $g_{\wAH, \epsilon}$ also lifts to be a $\phi$-metric for each $\epsilon$.

\subsection{Gluing construction}
The gluing is done on the connecting neck region
 $$
 \{\frac{1}{\delta}<|x'| <\frac{\delta}{\epsilon}\} \subset \widehat{AH}_{\epsilon}
 $$
and 
$$
\{\frac{\epsilon}{\delta} <|x|<\delta\} \subset M_{\epsilon}.
$$
Recall $\epsilon<\delta^{2}\ll1$.  We use the identification map $x'=\kappa(x)= x/\epsilon, \kappa^{*}(\alpha_{0})=\alpha'$ to identify $g_{\epsilon}$ and $g_{\widehat{AH}, \epsilon}$, and note that under this map $h_{\epsilon}$ and $h_{\epsilon'}$ are identified up to $O(\epsilon^{3})$. 

However in order to apply the analysis, we need to modify the space.
The total space is glued from two parts: $\cW_{0}$ corresponding to the Atiyah-Hitchin metrics in the center,  and $\cW_{1}$ corresponding to the multi-Taub-NUT singularities. Here $\cW_{0}$ and $\cW_{1}$ are essentially $\widehat{AH}_{\epsilon}$ and $M_{\epsilon}$ with the additional steps of compactification at infinity and singularity resolution at $\epsilon=0$.  The construction of $\cW_{1}$ and $\cW_{0}$ are similar to~\cite{SS} with a few modifications. See Figure~\ref{f:F1} for an illustration.

Before we talk about the construction, we will first briefly recall the idea of real blow ups. For a general introduction the readers are referred to~\cite{melrose1996differential} and for the specific construction for this paper we referred to~\cite[Section 3.2]{SS}. Given $X$ a manifold with corners, we call a submanifold $Y \subset X$  a $p$-submanifold if  around any point $q\in Y$, there is a neighborhood $\calU \subset X$ with a product
decomposition of the form $(Y \cap \calU) \times [0,\epsilon)^{\ell-\ell'} \times (-\epsilon, \epsilon)^{m-m'}$. 
If $Y$ is a $p$-submanifold of $X$, we define the blowup $[X; Y]$ as the disjoint union of $X \setminus Y$ and $SN^+ Y$, the inward-pointing
spherical normal bundle of $Y$, with a smooth structure
generated by the lifts of all smooth functions on $X$ and
the spherical normal coordinates around $Y$.  This blowup has a new front face equal to the set of all spherical
normal vectors.
Any vector field $V$ on $X$ lifts to a vector field on $[X;Y]$; this lift is nonsingular if and only if $V$ is tangent to $Y$. 
Blowup has the effect of making such tangent vector fields less degenerate. 

Now we introduce the construction.
The construction is first done on the double cover $\tcW_{i}, \ i=0,1$, and the final space $\cW$ is the quotient by the involution $\iota$. 
First we construct $\tcW_{0}$, which is defined as
\begin{equation}
\tcW_{0}=\{(z',\epsilon') \in \wAH_{\epsilon} \times [0,\epsilon_{0}): \frac{1}{\delta}<|\psi(z')| <\frac{\delta}{\epsilon'}\}.
\end{equation}

The construction of $\tcW_{1}$ is the same as in~\cite[Section 4.1]{SS}, obtained by pulling back the product  $[0,\delta^{2})_{\epsilon} \times M_{\epsilon}$   to the resolved base $\widetilde{B}$, where $\widetilde{B}$ is obtained through the blow up
\begin{equation}
\widetilde{B}=[\bar \RR^{3}_{x}\times [0, \delta^{2})_{\epsilon}; \{(0,0), \pm p_{i}\times \{0\}, i=1, \dots, k-1\}] \cap \{|x|>\epsilon/\delta\}
\end{equation}
which introduces $2k-1$ front faces. We then pull back the circle fibration $\phi: M_{\epsilon} \rightarrow \bar\RR^{3}\setminus\{|x|\leq \epsilon/\delta\}$ to $\widetilde B$, such that when restricting to each $\epsilon>0$ it is the same as the original fibration $\phi$. On the other hand, when restricting to $\epsilon=0$ it lifts to a circle bundle over the blown up space $[\bar \RR^{3}; \{0, \pm p_{i}\}]$.

We construct the total space by gluing $\tcW_{1}$ into $\tcW_{0}$, identifying the torus neighborhood using the lift of the map $\kappa$ described earlier, i.e. we identify the neighborhood 
$$
\{(z',\epsilon'): 0\leq \epsilon' \leq \epsilon_{0}, \delta^{-1}<|\psi(z')| <\delta (\epsilon')^{-1}\}\subset \tcW_{0}
$$
and
$$
\{(z,\epsilon): 0\leq \epsilon\leq \epsilon_{0}, \epsilon \delta^{-1}\leq |\phi(z)|\leq \delta\}\subset \tcW_{1}
$$
via the lift of the map $\phi(z)=x:=\psi(z') \epsilon, \epsilon=\epsilon'$, and we denote this map by $\kappa$. The total space is defined as
\begin{equation}
\tcW:=\tcW_{0}\coprod_{\kappa} \tcW_{1}, \ \cW:=\tcW/\iota.
\end{equation}
Here $\cW$ is a manifold with corners, where the boundary faces include: $I_{\infty}$ which corresponds to the fibration over the boundary of the compactified $\RR^{3}$ of $\cW_{1}$; $X_{0}$ that is the lift of the fibration to the boundary face created by blowing up $(0,0)$; $X_{i}, i=1, \dots k-1$, that are the lift of the fibration to the boundary faces created by blowing up $\{\pm p_{i}, 0\}$; $X_{ad}$ the remaining part of $\epsilon=0$.
See Figure~\ref{f:F1} for a picture of the total space.
\begin{figure}[ht]
		\centering
		\includegraphics[width=\textwidth]{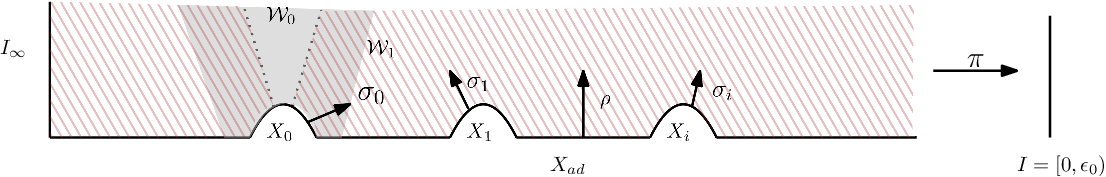}
		\caption{An illustration of the resolved total space $\cW$ obtained from gluing $\cW_{0}$ and $\cW_{1}$, with boundary faces and associated boundary defining functions on $\pi^{-1}(0)$.}
		\label{f:F1}
	\end{figure}

There is a projection map $\pi: \cW\rightarrow [0,\epsilon_{0})_{\epsilon}$, such that for each nonzero $\epsilon$ the fiber is an approximate ALF metric, while for $\epsilon=0$ the fiber is a manifold with corners. We also note that $\epsilon$ is the total boundary defining function, i.e. if we define $\sigma_{\nu}$ to be the boundary defining function of $X_{\nu}, \nu=0, 1, \dots, k-1$, and $\rho$ to be the boundary defining function of $X_{ad}$, then the product of all boundary defining functions on the front face is
\begin{equation}
\epsilon=\rho \sigma_{0}\dots \sigma_{k}.
\end{equation}

\subsection{Solving for the hyperK\"ahler triples}
By the gluing construction discussed above, we arrived at a family of approximate metrics, each living on the fiber $\pi^{-1}(\epsilon)$ of $\cW$. We will follow the set up in ~\cite{SS} to perturb these metrics and solve for the exact hyperK\"ahler triples on each fiber. Recall that every hyperK\"ahler metric on a 4-manifold gives a hyperK\"ahler triple $\bom =(\bom_{1}, \bom_{2}, \bom_{3})$ such that each $\omega_{i}$ is a symplectic form and they satisfy the following equation~\cite{Donaldson}
\begin{equation}
0=Q(\bom)_{jk}:=\bom_{j} \wedge \bom_{k} - \frac{1}{3}(\sum_{i=1}^{3} \bom_{i}^{2}) \delta_{jk}
\end{equation}
Conversely a hyperK\"ahler triple determines a hyperK\"ahler metric~\cite[Theorem 2.3]{SS}.

For an adiabatic Gibbons-Hawking ansatz in~\eqref{eq:ge}, the triple is given by the following $\bom_{\epsilon}=(\bom_{1},\bom_{2}, \bom_{3})$~\cite[(2.21)]{SS}
\begin{equation}\label{eq:getriple}
\begin{aligned}
\bom_{1}=\alpha\wedge \frac{dx_{1}}{\epsilon}+h_{\epsilon}\frac{dx_{2}\wedge dx_{3}}{\epsilon^{2}}\\
\bom_{2}=\alpha\wedge \frac{dx_{2}}{\epsilon}+h_{\epsilon}\frac{dx_{3}\wedge dx_{1}}{\epsilon^{2}}\\
\bom_{3}=\alpha\wedge \frac{dx_{3}}{\epsilon}+h_{\epsilon}\frac{dx_{1}\wedge dx_{2}}{\epsilon^{2}}
\end{aligned}
\end{equation}
which can be written as $\Omega_{\epsilon} + db$ where $b \in O(|\epsilon|^{3})$ and $\Omega_{\epsilon}$ is given by replacing $h_{\epsilon}$ by its leading term $1-\frac{\epsilon}{|x|} +\sum_{v=1}^{k-1}\frac{\epsilon}{|p_{v}|}$. 
Similarly for the branched cover $\wAH_{\epsilon'}$ the triple is given by $\bom_{AH, \epsilon'}=\Omega_{\epsilon'} + da$ where $a=O(e^{-|x'|})$ and $\Omega_{\epsilon'}$ is given by
\begin{equation}\label{eq:gAHtriple}
\begin{aligned}
\bom_{1}'=\alpha'\wedge dx_{1}'+h_{\epsilon}' dx_{2}'\wedge dx_{3}'\\
\bom_{2}'=\alpha'\wedge dx_{2}'+h_{\epsilon}' dx_{3}'\wedge dx_{1}'\\
\bom_{3}'=\alpha'\wedge dx_{3}'+h_{\epsilon}' dx_{1}'\wedge dx_{2}'\\
\end{aligned}
\end{equation}
with 
$$
h_{\epsilon}'=1-\frac{1}{|x'|} +\sum_{v=1}^{k-1}\frac{\epsilon}{|p_{v}|}.
$$

We will patch the two triples on each fiber to an initial approximation, then solve the equation and show that the solutions are uniformly controlled up to the $\pi^{-1}(0)$ face. This will give a family of smooth hyperK\"ahler triples hence the corresponding family of hyperK\"ahler metrics. To give a more precise statement, we use the following space:
\begin{definition}
Let $T_{\phi}(\cW/I)$ denote the subspace of the total $\phi-$tangent space on $\cW$ that is spanned by $\phi$-tangent vectors along the $\epsilon-$fiber direction. Similarly the $T^{*}_{\phi}(\cW/I)$ and $\Omega^{2}_{\phi}(\cW/I)$ denote the fiberwise $\phi-$cotangent space and (smooth) 2-forms.
\end{definition}
The symplectic forms will be in the space of $\Omega^{2}_{\phi}(\cW/I)$, and the operator $Q$ maps into the space $S_{0}^{2}\RR^{3}\otimes \Lambda^{4}T_{\phi}^{*}(\cW/I)$.  With these notations, now we state the theorem.
\begin{theorem}\label{thm:main}
Fix $\{0, p_{1}, \dots, p_{k-1}\}$ a set of distinct $k$ points. There exists $\epsilon_{0}>0$, such that on $\pi^{-1}([0,\epsilon_{0})) \subset \cW$,  there exists smooth triple $\bom_{Se}=(\bom_{1}, \bom_{2}, \bom_{3})$ with $\bom_{i}\in\Omega_{\phi}^{2}(\cW/I)$ satisfying $Q(\bom)=0$. The corresponding fiber metric $g_{Se, \epsilon}$ satisfies
\begin{enumerate}
\item $\{g_{Se, \epsilon}\}_{\epsilon\in I}$ is a family of smooth fiberwise $\phi$-metric on $\cW$;
\item $g_{Se, 0}|_{X_{0}}$ restricts to be the Atiyah--Hitchin metric, and $g_{Se, 0}|_{X_{i}}$ restricts to be the Taub--NUT metric for $i=1, \dots, k-1$.
\end{enumerate}
 \end{theorem}
\begin{proof}
The proof follows almost the same as in~\cite{SS}. Such a strategy has also been used in~\cite{MZ1, MZ2}. We first construct an approximate solution $\zeta_{0}$ for the hyperK\"ahler triple by patching the two model metrics together using a cut off function $\chi$ that is supported near $\cW_{0}$:
$$
\bom_{\chi}:=\chi \bom_{AH, \epsilon'} + (1-\chi) \bom_{\epsilon}
$$ 
where the two hyperK\"ahler triples are defined in~\eqref{eq:getriple} and~\eqref{eq:gAHtriple}. Note that the harmonic function in $\bom_{AH,\epsilon'}$ and $\bom_{\epsilon}$ agree up to order $\epsilon^{3}$, therefore the initial approximation solves the equation up to the following error 
$$
Q(\bom_{\chi})\in \epsilon^{3}C^{\infty}(\cW, S_{0}^{2}\RR^{3}\otimes \Lambda^{4}T_{\phi}^{*}(\cW/I))
$$ 
which also vanishes to infinite order away from $X_{\nu}$.

The second step is to solve the equation iteratively to $\epsilon^{\infty}$ error. This step  follows from the discussion in~\cite[Section 5]{SS}, where they solve the error iteratively near the corner of $X_{\nu}\cap X_{ad}$ on the two faces.  Recall that the linearized operator of $Q$ is given by 12 copies of scalar Laplacians by taking basis of self-dual 2-forms. When restricting to the solvability near $X_{\nu}$, note that the new component $AH$ is still a strongly ALF metric hence the analysis from~\cite{SS} carries through. The solvability near $X_{ad}$ is the same  as before. As a result we can iteratively construct a formal power series solution 
$\zeta = \bom_{\chi} + dc$ such that
$$
Q(\zeta) \in \dot C^{\infty} (\cW,S_{0}^{2}\RR^{3}\otimes \Lambda^{4}T_{\phi}^{*}(\cW/I))
$$
that is, the error vanishes to infinite order on every boundary face. 

The third step is to show there is a bounded inverse of the linearized operator on $\cW$. This step again relies on that one can patch inverses on $X_{\nu}$ and $X_{ad}$. Since the only difference in the structure of $X_{\nu}$ from~\cite{SS} is that $X_{0}$ is changed from $\overline{\calM_{2}^{0}}$ by $\overline{AH}$, the inverse is constructed in the same way as in~\cite[Theorem C.3]{SS} for a strongly ALF space. On the other hand, on $X_{ad}$ the inverse is constructed explicitly using Fourier series and Green's function where the only difference is the coefficient in $h_\epsilon$ which does not change the proof. As a result $\Delta_{\zeta}$ is an invertible operator on the domain constructed in~\cite[Theorem 6.6]{SS}.

Finally by implicit function theorem we get the exact solution $\omega_{Se}$ with $Q(\omega_{Se})=0$. The smoothness comes from the iterative series and invertibility of the linearized operator, and the restriction on $X_{\nu}$ follows because AH and TN are the 0-th order terms in the  metric expansion.
\end{proof}

As a corollary, such a family of metrics degenerates into the front face, and (under correct scaling and choice of base points) this gives Gromov-Hausdorff convergence to $AH$ and $TN$ respectively.

\begin{corollary}
\begin{enumerate}
\item 
Take $\{z_{\epsilon}\}$ satisfying $|\phi(z_{\epsilon})|<\frac{\epsilon}{\delta}$, then the sequence $(g_{Se, \epsilon}, z_{\epsilon})$ converges to $(g_{AH}, 0)$ in the Gromov-Hausdorff limit as $\epsilon\rightarrow 0$;
\item Take $\{z'_{\epsilon}\}$ satisfying $|\phi(z'_{\epsilon})-p_{i}|<\epsilon$, then $(g_{Se, \epsilon}, z'_{\epsilon})$ converges to $(g_{TN}, 0)$ in the Gromov-Hausdorff limit.
\end{enumerate}
\end{corollary}
\begin{remark}
Different types of gravitational instantons are expected to relate to each other, and similar results of this type where gravitational instantons degenerate into another type can be seen in~\cite{Auvray, Cherkis, LSZ, salm1, salm2}.
\end{remark}

\section{Non-holomorphic minimal spheres}
In~\cite{MW2} Micallef and Wolfson showed that in the Atiyah--Hitchin space there is a special sphere  that is stable minimal and not holomorphic with respect to any complex structures in the hyperK\"ahler rotation. Using the gluing result above, we extend this result to $D_{k}$ spaces constructed above. As a perturbation result, the existence can be proved for $D_{k}$ spaces where the Taub-NUT singularities are sufficiently far from the origin. 

\begin{theorem}\label{thm2}
There exists $\epsilon_{0}$, such that for every $\epsilon \in (0, \epsilon_{0})$ there exists a stable minimal sphere in $g_{Se, \epsilon}$ that is non-holomorphic with respect to any of the compatible complex structures.   
\end{theorem}
\begin{proof}
We start with the approximate metric $g_{0, \epsilon}$ constructed via patching the $g_{\wAH, \epsilon}$ and $g_{\epsilon}$. Since the minimal sphere in $g_{\wAH}$ is compact, so by shrinking $\epsilon_{0}$ if necessary, there exists a sphere in each fiber $\Sigma_{\epsilon} \subset (AH|_{|x'|<\delta/\epsilon}, g_{\wAH, \epsilon})\subset (\pi^{-1}(\epsilon), g_{0, \epsilon})$ for every $\epsilon\in (0, \epsilon_{0})$. Note that this minimal sphere $\Sigma$ satisfies $[\Sigma]\cdot [\Sigma]=-4$ which is not $2*\text{genus}(\Sigma)-2$, hence it cannot be holomorphic in a hyperK\"ahler manifold via adjunct formula. Since the minimal sphere is stable in the Atiyah--Hitchin space~\cite{TsaiWang}, we can apply White's implicit function theorem~\cite{White} (which does not require the ambient manifold to be compact) to the actual hyperK\"ahler metric $g_{SE, \epsilon}$ perturbed from $g_{0, \epsilon}$, we get a stable minimal non-holomorphic sphere in the $D_{k}$ space. 
\end{proof}
 
\begin{corollary}\label{cor1}
There is an open neighborhood in the moduli space of hyperK\"ahler metrics on $K3$, such that each metric in this neighborhood contains a stable minimal sphere that is not holomorphic with respect to any compatible complex structures. 
\end{corollary}
\begin{remark}
The statement follows from \cite[Theorem 7.1]{Foscolo} via perturbation, however this construction provides a larger neighborhood and gives more explicit examples compared to the statement in~\cite{MW1}. 
\end{remark}
\begin{proof}
The proof follows exactly the same as in~\cite[Theorem 7.1]{Foscolo} by considering any $D_{k}$ metric constructed above as a building block in the collapsing sequence. Since the approximate metric $g_{\epsilon}$  constructed using $D_{k}$ bubbles contains the special sphere, by the same perturbation argument in the previous theorem we obtain the existence of a non-holomorphic minimal sphere in a neighborhood.  
\end{proof}

\begin{remark} We do not know whether the sphere we constructed is actually minimizing in those new ALF spaces and K3 surfaces. The model sphere in AH is area-minimizing, and in fact, calibrated~\cite{TsaiWang}. It is an interesting question whether any of such properties will be preserved after the gluing construction.
\end{remark}

\bibliographystyle{amsalpha}
\bibliography{ALF}

\end{document}